\documentclass[12pt,reqno]{amsart}
\usepackage[margin=3cm]{geometry}

\usepackage{amsmath, amssymb, amsfonts, tikz,enumerate, graphicx, textcomp, caption, wrapfig, amsthm, todonotes, verbatim, cleveref, caption, float, mathabx, flafter, url, adjustbox}
\usetikzlibrary{calc, arrows}
\usepackage[procnames]{listings}
\usepackage{color, subfig}
\usepackage[section]{placeins}

\usepackage{color}
\definecolor{purple}{rgb}{0.5,0,1}

\definecolor{keywords}{RGB}{255,0,90}
\definecolor{comments}{RGB}{0,0,113}
\definecolor{red}{RGB}{160,0,0}
\definecolor{green}{RGB}{0,150,0}
\lstdefinelanguage{Magma}%
  {%
   otherkeywords={:=,+:=,-:=,*:=},%
   procnamekeys={function,func,intrinsic,procedure,proc},%
   morekeywords={true,false},%
   morekeywords=[2]{adj,and,cat,cmpeq,cmpne,diff,div,eq,ge,gt,in,is,join,le,lt,%
          meet,mod,ne,notadj,notin,notsubset,or,sdiff,subset,xor},%
   morekeywords=[3]{assigned,break,by,case,catch,continue,declare,default,%
          delete,do,elif,else,end,eval,exists,exit,for,forall,fprintf,if,local,%
          not,print,printf,quit,random,read,readi,repeat,restore,save,select,%
          then,time,to,try,until,vprint,vprintf,vtime,when,where,while},%
   morekeywords=[4]{clear,forward,freeze,iload,import,load},%
   morekeywords=[5]{assert,assert2,assert3,error,require,requirege,requirerange},%
   morekeywords=[6]{car,comp,cop,elt,ext,frac,hom,ideal,iso,lideal,loc,map,%
          ncl,pmap,quo,rec,recformat,rep,rideal,sub},%
   morekeywords=[7]{AbelianGroup,AdditiveCode,AffineAlgebra,Algebra,%
          AssociativeAlgebra,Character,CliffordAlgebra,Design,Digraph,%
          ExtensionField,FPAlgebra,FiniteAffinePlane,FiniteProjectivePlane,%
          Graph,Group,GroupAlgebra,IncidenceStructure,LieAlgebra,LinearCode,%
          LinearSpace,MatrixAlgebra,MatrixGroup,MatrixRing,Monoid,%
          MultiDigraph,MultiGraph,NearLinearSpace,Network,PartialMap,%
          PermutationGroup,PolycyclicGroup,QuaternionAlgebra,Semigroup,%
          ZModule},%
   morekeywords={[8]function,func,intrinsic,procedure,proc,return},%
      sensitive,%
      morecomment=[l]//,%
      morecomment=[s]{/*}{*/},%
      morecomment=[s]{\{}{\}},%
      morestring=[b]"%
  }[keywords,procnames,comments,strings]%
\usepackage{tikz-cd}\tikzset{node distance=2cm, auto}
\DeclareMathOperator{\Aut}{Aut}
\DeclareMathOperator{\Hom}{Hom}
\DeclareMathOperator{\Jac}{Jac}
\DeclareMathOperator{\Alt}{Alt}

\newcommand{\C}{\mathbb{C}}
\newcommand{\Z}{\mathbb{Z}}

\newcommand{\Q}{\mathbb{Q}}

\newcommand{\n}{\newline}
\newcommand{\mc}{\mathcal}

\newcommand{\bb}{\mathbb}
\renewcommand{\P}{\mathbb{P}}

\definecolor{codegray}{gray}{0.9}

\newtheorem{theorem}{Theorem}
\newtheorem*{thm*}{Theorem}
\newtheorem*{proposition}{Proposition}

\newtheorem*{lemma*}{Lemma}
\newtheorem*{qlemma*}{``Lemma"}

\newtheorem*{conjecture*}{Conjecture}

\newtheorem*{question}{Question}
\theoremstyle{definition}
\newtheorem{defn}{Definition}

\theoremstyle{remark}
\newtheorem*{remark}{Remark}
\newtheorem*{answer}{Answer}

\newcommand{\DD}{\Delta\kern -8.3pt {\diamond} \kern -4.5pt \cdot \:}

\title[Automorphisms of Abelian Varieties and Principal Polarizations]{Automorphisms of Abelian Varieties \\ and Principal Polarizations}
\author{Dami Lee}  
\address{University of Washington, Seattle, WA 98195}
   \email{damilee@uw.edu}  

   \author{Catherine Ray}
   \address{Northwestern University, Evanston, IL 60208}
   \email{cray@math.northwestern.edu}

\begin{document}

\maketitle

\begin{abstract} %Abelian varieties with several principal polarizations are incredibly rare. 
The Narasimhan-Nori conjecture asks for a closed formula for the number of non-isomorphic principal polarizations of any given abelian variety. In this paper, we introduce a new algorithm that gives a lower bound on the number of non-isomorphic principal polarizations on any given abelian variety. We show, for example, that the Jacobian of the genus four underlying curve of Schoen's I-WP minimal surface has at least 9 non-isomorphic principal polarizations. We also explore the Jacobians of Klein's quartic, Fermat's quartic, Bring's curve, and more.   
\end{abstract}

\tableofcontents

\section{Introduction}

Polarizations of abelian varieties were introduced by Weil as an analogue of orientations of manifolds. They are embeddings of abelian varieties into projective space. There is a special kind of polarization called a principal polarization -- all polarizations are induced by these up to isogeny. The fact that an abelian variety admits only a finite number of isomorphism classes of principal polarizations was established by Narasimhan-Nori. In \cite{nn}, they pose the problem of finding a closed formula for the number of principal polarizations of any given abelian variety over any field, which is still unsolved. 

There has been a lot of prior work on the Narasimhan-Nori conjecture, exposited in Section ~\ref{sec:prior}. In this paper, we introduce an entirely different method of finding different principal polarizations on Jacobians, exposited in Section ~\ref{sec:find}, which treats both simple and non-simple cases in characteristic 0. We use this to give lower bounds on the number of non-isomorphic principal polarizations. Our method works as follows. For curves that are cyclic covers over $\C\bb{P}^1$, we can compute exact period matrices, exposited in Section ~\ref{sec:cyclicperiod}. Given any period matrix, we introduce new code to compute principal polarizations on the corresponding abelian variety. Then, we implement a modification of the psuedocode of Bruin-Jeroen-Sijsling \cite{numerical} to compute, for each found principal polarization, the automorphism group of the given Jacobian which fixes that polarization. If the automorphism groups are different, then the principal polarizations are non-isomorphic. This gives us a lower bound on the number of different principal polarizations for a given Jacobian. We denote by $\pi(X)$ the number of non-isomorphic principal polarizations on a variety $X.$
%The fact that numerical computations with period matrices, which we use for surfaces which are not cyclically branched covers of $\C\P^1,$ give rigorous results relies on the brilliant work of Costa-Mascot-Sijsling-Voight (\cite{rigor} Prop 6.1.1), which we discuss in Section ~\ref{sec:cert}.

For example, our code yields the following result. 

\begin{theorem} \label{IWP} Let $\pi(X)$ denote the number of non-isomorphic principal polarizations on any given variety $X$. Let I-WP denote the underlying genus four curve of Schoen's I-WP minimal surface, and $\Jac(\text{I-WP})$ its Jacobian variety, then $$\pi(\Jac(\text{I-WP})) \geq 9.$$\end{theorem}

This is a remarkable result, especially since the variety $\Jac(\text{I-WP})$ itself factors into a product of 4 elliptic curves\footnote{This is because $\text{End}(\Jac(\text{I-WP})) \simeq M_4(K)$, where $K$ is imaginary quadratic. Thus, $\Jac(\text{I-WP})$ is the product of elliptic curves with CM by $K$.}, so the remaining principal polarizations must come from interesting new cycles in the product of these elliptic curves. Other results found by applying our technique include $\pi(A)$ and $|\Aut(A)|$ of Klein's quartic, Fermat's quartic, and Bring's curve (Table ~\ref{table:tablelabel}, Section ~\ref{sec:examples}).

We exposit cyclic covers of $\C\P^1$ that allow us to compute exact period matrices for simplicity of exposition, our method also works for nonexact period matrices.

In Section ~\ref{sec:questions}, we also answer the following questions on Jacobians with multiple principal polarizations by providing counterexamples.

\subsection{Prior Work}
\label{sec:prior}
Previous work toward the Narasimhan-Nori conjecture can be divided into the simple and non-simple cases. %An abelian variety is \textbf{simple} if it is not isogenous to a product of abelian varieties of lower dimension. 
It can further be divided into characteristic 0 and characteristic $p$ cases. Our main goal is to understand cases of simple varieties in characteristic 0. 

In Theorem 1.5 \cite{several}, Lange establishes for simple varieties that $|\Aut(A)| = \pi(A)$ with certain restriction conditions and equivalence relations. One could in principle compute $|\Aut(A)|$ by hand, however, it is computationally infeasible. 

%\begin{remark} His bijection between sets is induced by considering an element in the Neron-Severi group of $A$, and representing such elements as endomorphisms of $A$ preserved under the Rosati involution with respect to a chosen polarization. Since this bijection is induced by a principal polarization, one must know that one exists to implement this theorem. \end{remark}

In Theorem 3.1 \cite{several}, Lange further establishes bounds on $\pi(A)$ in terms of the class group of $\text{End}_{\Q}(A)$, if $\text{End}_{\Q}(A)$ is a totally real number field $K$ (and thus the variety $A$ is simple). More recently, Lange treated the non-simple case of products of elliptic curves without complex multiplication in Theorem 3.5 \cite{newlange}. He does so by giving an interpretation of the number of principal polarizations in terms of class numbers of definite Hermitian forms.

%\begin{remark} The main idea of the proof of 3.5 is that the canonical principal polarization of X induces a bijection between P(X) and the set of equivalence classes of symmetric automorphisms of X. Via the analytic representation of X and a suitable choice of bases this set can be considered as a set of equivalence classes of Hermitian forms. \end{remark}

Given an abelian variety $A$ and a curve $C$, an isomorphism $A \simeq \Jac(C)$ induces a polarization on $A$ called the canonical principal polarization [Defn ~\ref{canpp}]. Note that this polarization does depend on the specified curve $C$. In particular, if $A \simeq \Jac(C) \simeq \Jac(C')$ is the Jacobian of multiple different non-isomorphic curves $C \not\simeq C'$, there can be several ``canonical" principal polarizations on a Jacobian.

%The canonical principal polarization is a principal polarization on an abelian variety $A$, induced by the isomorphism $A \simeq \Jac(C)$ for a specified curve $C$ . It is very important to note that there is only a canonical principal polarization on $A$ if a curve $C$ is specified. If $C$ is not specified, there is no canonical choice -- knowing that $A$ is in the image of the functor $\Jac$ is not enough. More specifically,  Its canonicality only refers to the fact that the principal polarization on $A$ induced by a \textit{specified} curve $C$ is canonical.

All other previous works known to the authors on finding multiple principal polarizations on abelian varieties have been done by finding two non-isomorphic curves with the same (unpolarized) Jacobian. Therefore, their associated canonical polarizations must be different by Torelli theorem. Otherwise, the curves would be isomorphic. All papers that we know of using this technique do so only in characteristic $p$. The papers using this technique discuss the case of \textit{non-simple} Jacobians of curves of genus two \cite{iko} and three \cite{brock}.

This technique is again used by E. Howe \cite{howe1} and \cite{howe2} which gives examples of non-isomorphic genus two curves with the same \textit{simple} Jacobian. He finds such examples in characteristic $p$ through isogeny classes of abelian varieties which correspond to special Weil numbers, an application of the Honda-Tate method. 

\section*{Acknowledgements} 
The authors would like to thank \texttt{Magma} and \texttt{Sage} contributors Edgar Costa, Nicolas Mascot, John Voight, and above all Jeroen Sijsling, who generously offered incredibly detailed and consistent help in computing the automorphism groups of Jacobians. Lee would also like to thank Matthias Weber for his guidance at the beginning of this project. This material is based upon work supported by the National Science Foundation under Grant No. DMS-1440140 while Lee was in residence at the Mathematical Sciences Research Institute in Berkeley, California, during the Fall 2019 semester. Ray is partially supported the National Science Foundation GRFP under Grant Number DGE 1842165.

\section{Objects Of Our Interest}
\label{sec: dthesis}
We use computational techniques from \cite{dthesis} to compute explicit period matrices. These techniques apply to curves that are cyclically branched over $\mathbb{C}\mathbb{P}^1.$ In this section, we summarize Chapter 3 of \cite{dthesis} and devote this section to providing background material on cyclic covers over $\mathbb{C}\mathbb{P}^1.$ We discuss the construction of such curves, various cone metrics that arise from the construction, and how to compute period matrices. 

\subsection{Construction of Cyclic Covers over $\mathbb{C}\mathbb{P}^1$}
First, we discuss the topological construction of cyclic covers over $\mathbb{C}\mathbb{P}^1.$ This will naturally yield cone metrics on the curves.

\begin{defn} We say that a curve $C$ is a $d$-fold cyclic cover over $\mathbb{C}\mathbb{P}^1$ if $C / (\Z/ d \Z) = \mathbb{C}\mathbb{P}^1.$ \end{defn}

We construct such curves with the given data: let $p_i, \ldots , p_n \in \mathbb{C}\mathbb{P}^1$ be $n$ distinct points. Let $Y := \mathbb{C}\mathbb{P}^1 \backslash \{p_1, \ldots, p_n\}$ and let $\gamma_i$ be a branch cut from $p_i$ to some $q \in Y$ so that $\gamma_i$ are mutually disjoint. For each $i,$ assign $d_i \in \{1, \ldots, d - 1\}$ and call it the \textbf{branching index at} $p_i.$ Let $d$ be the degree of the covering map and use $j$ to label $Y_1, \ldots , Y_d.$ For each $i$ and $j,$ we identify the ``left side'' of $\gamma_i$ of $Y_j$ to the ``right side'' of $\gamma_i$ of $Y_{j + d_i \pmod d}.$ We denote such a covering $C$ by a $n$-tuple $d (d_1, \ldots , d_n).$

\begin{remark} A covering $d (d_1, \ldots , d_n)$ is uniquely defined up to homeomorphism. That is, the construction only depends on $d_i$ and is independent of $p_i,$ $\gamma_i,$ and $q.$ We further assume that $\sum\limits_{i=1}^n d_i \equiv 0 \pmod d$ and $\gcd (d_1, \ldots, d_n) = 1.$ The former guarantees that the covering is closed and the latter guarantees that the covering is connected. Then one can compute the genus of the curve by Riemann-Hurwitz formula and $g(C) = \frac{d (n-2)}{2} + 1 - \frac{1}{2} \sum\limits_{i=1}^n \gcd(d,d_i).$ 
\end{remark}

\subsubsection*{Branching indices on Fermat's quartic}
As a running example, we will look at Fermat's quartic that was studied as a cyclic cover in \cite{dami}. The curve is invariant under an order-eight rotational symmetry (Figure~\ref{fig:125}). Its quotient under this rotation is $\mathbb{C}\mathbb{P}^1$ and the covering is defined as $8 (1, 2, 5).$

\begin{figure}[htbp] 
   \centering
   \includegraphics[width=2in]{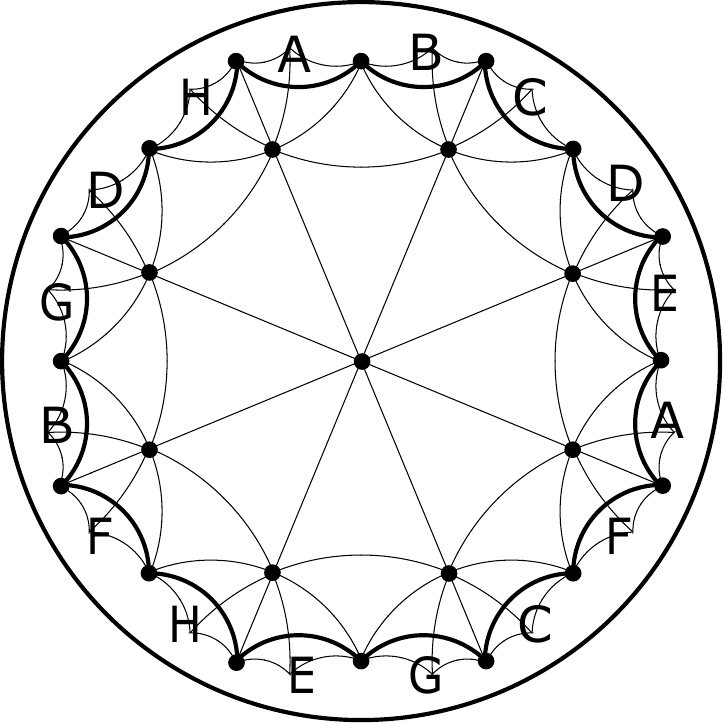} 
  \caption{Hyperbolic tessellation on $8(1, 2, 5).$ Reprinted from \cite{dami}.}
  \label{fig:125}
\end{figure}

%The euclidean triangles on the fundamental piece of $\Pi$ have a one-to-one correspondence to the hyperbolic triangles.

For each $i,$ there are $\gcd(d, d_i)$ preimages $\widetilde{p_i}$ of $p_i$ on $C.$ Hence, each $\widetilde{p_i}$ is a non-trivial cone point. To pin down a holomorphic 1-form on $C,$ we find cone metrics on $\mathbb{C}\mathbb{P}^1$ and pull-back to its covering. We say that a cone metric on $\mathbb{C}\mathbb{P}^1$ is \textbf{admissible} if its pullback yields a flat structure on $C.$ The following proposition is a version of Gauss-Bonnet theorem.

\begin{proposition} Given a compact Riemann surface of genus $g$ with a cone metric, let $p_1, \ldots, p_n$ be distinguished points with respective cone angles $\theta_i.$ Then $\sum\limits_{i=1}^n \theta_i = 2 \pi (2 g - 2 + n).$
\end{proposition}

Specifically, the sum of cone angles on a genus zero curve is $2 \pi (n - 2).$ Given branching indices such that $\sum\limits_{i=1}^n d_i = d (n - 2),$ we get a cone metric on $\mathbb{C}\mathbb{P}^1$ with cone angles $\frac{2 \pi d_i}{d}$ at each $p_i.$ For example, by putting a cone metric on the quotient sphere where the cone angles are $\frac{1 \pi}{8}, \frac{2 \pi}{8},$ and $\frac{5 \pi}{8}$ as in Figure~\ref{fig:125_flat}, one can see that the identification of edges are by translations. In other words, this gives rise to a translation structure on the eightfold cover of the sphere. Moreover, we obtain a holomorphic 1-form with one order-4 zero.

\begin{figure}[htbp]
   \centering
   \includegraphics[width=2.5in]{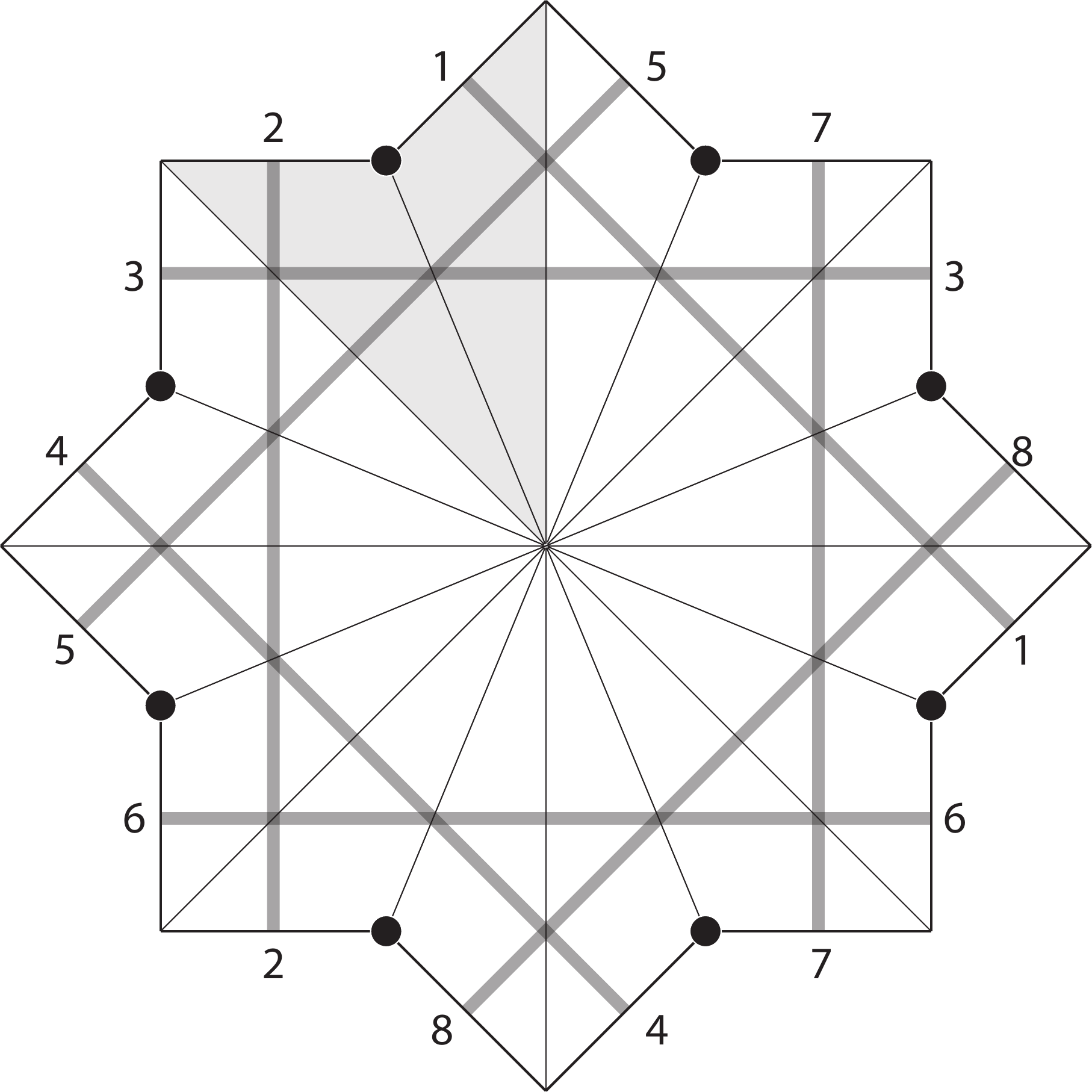} 
  \caption{A flat structure $\omega_1$ on $8(1, 2, 5)$ with one marked zero. Modified from \cite{dami}}
  \label{fig:125_flat}
\end{figure}

Our goal is to find other admissible cone metrics on the sphere that yield different translation structures (linearly independent 1-forms) on the same curve. We claim that the cone metric given by cone angles $\frac{2 \pi a_i}{d}$ where $a_i \equiv d_i \pmod d$ for all $i$ is admissible (Theorem 3.4, \cite{dthesis}). An easier way of finding admissible cone metrics is given by the following notion of multipliers.

\begin{defn} Given branching indices $d (d_1, \ldots , d_n),$ we say $a \in \{1, \ldots, d - 1\}$ is a \textbf{multiplier} if the cone metric given by cone angles $\frac{2 \pi}{d} (a \cdot d_1 \pmod d, \ldots , a \cdot d_n \pmod d)$ is admissible. 
\end{defn}

Theorem 3.5, \cite{dthesis} proves that for $n = 3,$ there are exactly $g$ multipliers. In other words, one achieves a basis of holomorphic 1-forms via multipliers. We denote the multipliers by $a_i,$ hence the admissible cone angles are given by $$\frac{2\pi}{d}(a_1 d_1, a_1 d_2, a_1 d_3), \ldots , \frac{2\pi}{d}(a_g d_1, a_g d_2, a_g d_3).$$ 

With admissible cone metrics, we can algebraically describe a curve by its plane curve model. For covers branched over three points, we choose $p_i = 0, 1, \infty$ and describe it as $y^d = x^{d_1} (x-1)^{d_2}.$ However, a plane curve model for a given curve is not uniquely defined. The allowed choices of models is discussed in Section 3.2 of \cite{dthesis}. However, our codes do not differentiate the different plane curve models.

\subsubsection*{Admissible cone metrics on $8 (1, 2, 5)$} Given branching indices $8 (1, 2, 5),$ multipliers 1, 2, and 5 give rise to cone metrics with cone angles $\frac{2 \pi}{8}(1, 2, 5),$ $\frac{2 \pi}{8}(2, 4, 2),$ and $\frac{2 \pi}{8}(5, 2, 1),$ respectively. These cone metrics yield a basis of holomorphic 1-forms with the following divisors: $$(\omega_1) = 4 \widetilde{p_3}, \qquad (\omega_2) = \widetilde{p_1} + \widetilde{p_2}_1 + \widetilde{p_2}_2 + \widetilde{p_3}, \qquad (\omega_3) = 4 \widetilde{p_1}.$$

Figure~\ref{fig:flat_rs2} represents $\omega_2$ given by cone angles $\frac{2 \pi}{8}(2, 4, 2).$ The four simple zeros are marked on the figure.

\begin{figure}[htbp]
   \centering
   \includegraphics[width=4in]{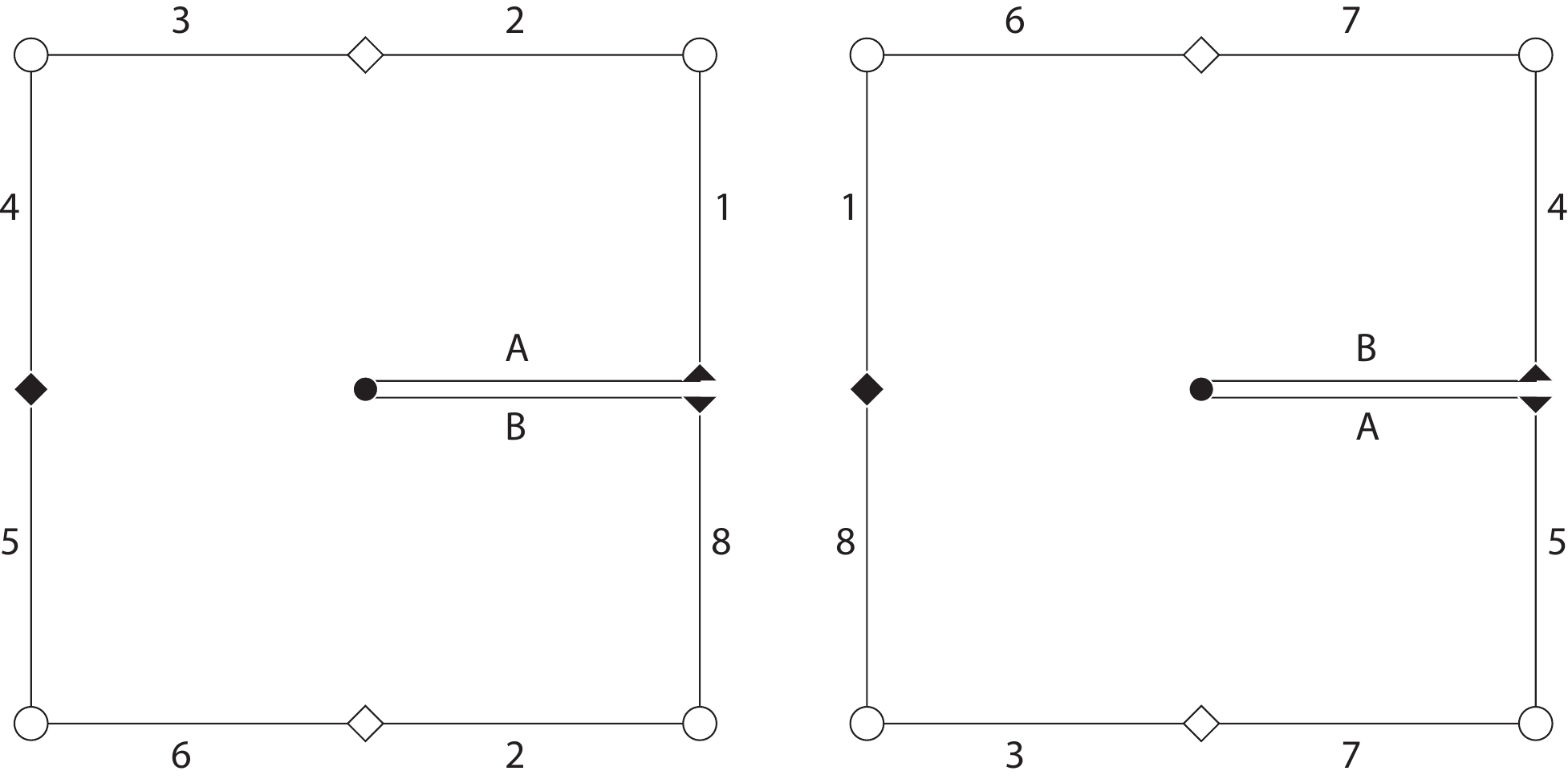} 
  \caption{A flat structure $\omega_2$ on $8(1, 2, 5)$ with four simple zeros. Modified from \cite{dthesis}}
  \label{fig:flat_rs2}
\end{figure}

%\begin{remark} Note that the multipliers preserve the labeling of edges in both Figure~\ref{fig:125_flat} and Figure~\ref{fig:flat_rs2}: 1,2,3,4,5,6,2,8,4,7,6,1,8,3,7,5.\end{remark}

%\begin{remark} Note that 3, 6, and 7 are not multipliers as their corresponding forms do not yield canonical divisors. On the other hand, 4 is not a multiplier as the cone metric derived from cone angles $\frac{2 \pi}{8} (4, 0, 4)$ is not admissible. This induces a meromorphic form whose divisor is $3 \widetilde{p_1} - \widetilde{p_2}_1 - \widetilde{p_2}_2 + 3 \widetilde{p_3}.$ \end{remark}

%\begin{remark} The geometric representation of $\omega_3$ can also be achieved by reflecting $\frac{\pi}{8}(5, 2, 1)$ triangles. However, we omit the corresponding figure. \end{remark}

\subsubsection*{Computing the period matrix of cyclic covers}
\label{sec:cyclicperiod}
In this section, we use the flat structure of a surface to compute the period matrix of a given curve. We will look at the simplest case where $n = 3$ and $d_1 = 1.$ Then, since $\sum d_i = d,$ a cone metric with cone angles $\frac{2 \pi}{d}(d_1, d_2, d_3)$ is admissible. $Y$ is topologically equivalent to a doubled triangle with angles $\frac{2 \pi}{d}(d_1, d_2, d_3)$ so we construct $C$ with $d$ copies of $Y,$ which yields a flat structure on $C.$ 

The identification of edges on Figure~\ref{fig:125_flat} are via Euclidean translations, which verifies that the cone metric is admissible. Translations yield closed cycles on the surface from which we get a homology basis with the following intersection matrix

$$\textrm{int}_1 = \begin{pmatrix} 0 & 1 & 1 & 0 & 0 & 0 \\
 -1 & 0 & 1 & 1 & 0 & 0 \\
 -1 & -1 & 0 & 1 & 1 & 0 \\
 0 & -1 & -1 & 0 & 1 & 1 \\
 0 & 0 & -1 & -1 & 0 & 1 \\
 0 & 0 & 0 & -1 & -1 & 0 \end{pmatrix}.$$

Then the period matrix is computed as follows: 
$$\Pi_1 = \left(
\begin{array}{cccccc}
 1 & e^{\frac{\pi i}{4}} & i & e^{\frac{3 \pi i}{4}} & -1 & e^{-\frac{3 \pi i}{4}} \\
 1 & i & -1 & -i & 1 & i \\
 1 & e^{-\frac{3 \pi i}{4}} & i & e^{-\frac{\pi i}{4}} & -1 & e^{\frac{\pi i}{4}} 
\end{array}
\right).$$

In general, from cyclicity we get 

$$\Pi = \left(
\begin{array}{ccccc}
 1 & e^{\frac{2 \alpha_1 \pi i}{d}} & e^{\frac{4 \alpha_1 \pi i}{d}} & \cdots & e^{\frac{2 (2 g - 1) \alpha_1 \pi i}{d}} \\
 1 & e^{\frac{2 \alpha_2 \pi i}{d}} & e^{\frac{4 \alpha_2 \pi i}{d}} & \cdots & e^{\frac{2 (2 g - 1) \alpha_2 \pi i}{d}} \\
 \vdots\\
 1 & e^{\frac{2 \alpha_g \pi i}{d}} & e^{\frac{4 \alpha_g \pi i}{d}} & \cdots & e^{\frac{2 (2 g - 1) \alpha_g \pi i}{d}} \\
\end{array}
\right)$$
where $\alpha_i$ are the multipliers.

%In \cite{dthesis}, Lee computes the period matrix by choosing a symplectic homology basis. In \cite{dami}, Lee shows that the 8(1,2,5) covering has a geometric realization as a quotient of a triply periodic polyhedral surface where the hyperbolic structure of the underlying curve is induced from the polyhedral cone metric. A triply periodic polyhedral surface is a polyhedral surface that is invariant under a rank-three lattice in $\R^3.$ Figure~\ref{fig:funda} is the genus-three quotient of the polyhedral surface via its lattice of translations. 

%\begin{figure}[htbp]
%   \centering
%   \includegraphics[width=2.5in]{figures/funda.png} 
%  \caption{Geometric realization of 8(1,2,5) as a quotient of a triply periodic polyhedral surface. Modified from \cite{dami}}
%  \label{fig:funda}
%\end{figure}

%$$\textrm{int}_2 = \begin{pmatrix} 0 & 0 & 0 & 1 & 0 & 0 \\
% 0 & 0 & 0 & 0 & 1 & 0 \\
% 0 & 0 & 0 & 0 & 0 & 1 \\
% -1 & 0 & 0 & 0 & 0 & 0 \\
% 0 & -1 & 0 & 0 & 0 & 0 \\
% 0 & 0 & -1 & 0 & 0 & 0\end{pmatrix}$$
 
\begin{remark} In Theorem 5.6 of \cite{dthesis}, Lee computes the period matrix of Fermat's quartic by choosing a symplectic homology basis. This yields the following period matrix

$$\Pi_2 = (A|B) = \begin{pmatrix}  1 - i& -\frac{1 + i}{1 + \sqrt{2}}& \frac{1 + i}{1 + \sqrt{2}}& 1 + i& \sqrt{2}& 2 - \sqrt{2} \\  -2i& 2i& 2i& 2i& -2& -2\\ -1 - i & (1 - i)(1 + \sqrt{2})& (-1 + i)(1 + \sqrt{2})& 1 - i& i\sqrt{2}& i(-2 - \sqrt{2})  \end{pmatrix} $$ and $$\tau = (A^{-1}B) = \begin{pmatrix}i & \frac{1 + i}{2} & \frac{1 + i}{2}\\
\frac{1 + i}{2} & i & \frac{1 + i}{2}\\
\frac{1 + i}{2} & \frac{1 + i}{2} & i\end{pmatrix}.$$\end{remark}

 \section{Programmatically Computing Principal Polarizations of Abelian Varieties over $\mathbb{C}$}

\label{sec:computing}

\noindent All of our code is available at \begin{center}\url{https://github.com/catherineray/aut-jac}.\end{center} 

\subsection{Introduction to Polarizations: From Theory to Code}
\label{sec:intropol}
The notion of a polarization of an abelian variety has many faces. If a complex torus has a polarization, it is an abelian variety.

\begin{defn} A \textbf{polarization} of a complex torus $X$ is an embedding $j: X \to \P^N$ for large enough $N$. \end{defn}

We can understand this embedding $j$ as a map $$p \mapsto [a_1(p) : \cdots : a_{N-1}(p)],$$ where $a_i$ are a chosen generating set of global sections of a line bundle $\mc{L}$ on $X$. 

\begin{defn}  A polarization $c_1(\mc{L})$ of $X$ is called \textbf{principal} if $\mc{L}$ has only one section up to constants, i.e. $\dim H^0(X, \mc{L}) = 1$. Here, $c_1$ denotes the first chern class. \end{defn} 

As a motivational theorem:

\begin{thm*} [\cite{bl} 4.1.2] Every polarization is induced by a principal polarization via an isogeny. \end{thm*}

By Narasimhan-Nori \cite{nn}, there are only finitely many principal polarizations on a variety $X$ which are irreducible and smooth. As a corollary, only finitely many curves may have the same Jacobian, because each non-isomorphic curve gives a non-isomorphic principal polarization on its Jacobian.

\vspace{+5pt}
\noindent We end this introduction to polarizations with a special kind of principal polarization. 

\begin{defn} \label{canpp} The \textbf{canonical principal polarization} of an abelian variety $A \simeq \Jac(C)$, isomorphic to the Jacobian of a specified curve $C$, is a symplectic form on $\Jac(C)$ induced the intersection form $Q: H_1(C; \mathbb{Z}) \otimes H_1(C; \mathbb{Z}) \to\mathbb{Z}$. This is because $H^2(\Jac(C); \mathbb{Z}) \simeq \Hom(\bigwedge H_1(C; \mathbb{Z}); \mathbb{Z})$. \end{defn}

\subsection{Finding Principal Polarizations} 
\label{sec:find}

We begin with a representation of our abelian variety as $A := \C^g/\Pi\Z^{2g}$. Then $\Lambda$ is the associated lattice spanned by the columns of $\Pi$. Thus, we have a distinguished basis for the homology of $A$, corresponding to the columns of $\Pi$. 

\vspace{+5pt}

\textbf{Algorithm:} Compute many principal polarizations on a given abelian variety $A$.

\textit{Input:} An abelian variety $A := \C^g/\Pi\Z^{2g}$, where $\Lambda$ is the associated lattice to the period matrix $\Pi$. \n
$\text{}$ $\hspace{2mm}$\textit{Output:} Many principal polarizations on $A$.
\begin{enumerate} 

\item The \texttt{Magma} function \texttt{FindPolarizationBasis} determines all integral alternating pairings $E$ on the homology, i.e., $E\in \Alt^2(\C^g, \Z)$, for whose real extension we have:

$$E (i v, i w) = E (v, w)$$
This is a basis of alternating forms $\{E_i\}$. (For more context, see \cite{bl} 2.1.6.)
\item Check that $E$ is positive-definite.
\item Try some small combinations and sees if $E_i$ actually gives a pairing with determinant 1 indicating that $E_i$ is a principal polarization. If so, it returns $E_i$. This gives us a set $\{E_k\}$ of integral pairings on the homology.
\item For each $i$, we rewrite these pairings in a symplectic basis. That is, we find a basis of $\Lambda$ in which $$E_i = \begin{pmatrix} 0 & D \\ -D & 0 \end{pmatrix}$$ where $D = \text{diag}(d_1, ..., d_g)$ which we may do by the elementary divisor theorem (Section 3.1, \cite{bl}). 

\end{enumerate} 

This does nothing but modify the (homology) basis of $\Lambda$. Multiplying $\Pi$ on the right with this integral matrix, we get a new period matrix $Q$ whose columns span exactly the same lattice but for which the standard symplectic pairing $E$ is actually the Chern class of a line bundle. This is often called the Frobenius form of the period matrix $\Pi$.

\section{Studying Principal Polarizations via Automorphisms of the Jacobian}

The algorithm we use for calculating the symplectic automorphism groups of curves comes from the psuedocode of Section 4 of \cite{numerical}. 

\begin{defn} Let $a$ be a polarization of $X$. We call $\Aut(X, a)$ a \textbf{symplectic automorphism group} of $X$, as it respects the symplectic form $a$. 
\end{defn}

\begin{defn} We say two principal polarizations $a_1$ and $a_2$ on $A$ are \textbf{auto-equivalent} if and only if $\Aut(A, a_1) \simeq \Aut(A, a_2)$. \end{defn}

Our program produces many auto-equivalent principal polarizations.  Note that auto-equivalence is a weaker notion of equivalence than analytic equivalence, as we will show in Section ~\ref{sec:questions}. 

If the abelian variety is indeed a Jacobian, this method will in practice return at least enough polarizations to find a canonical principal polarization. It is an unsolved problem to find \textit{all} possible principal polarizations associated to a given abelian variety, called ``explicit Narasimhan-Nori."

\subsection{Certifying Heuristic Methods}
\label{sec:cert}

We must certify that numerically computed endomorphisms of a Jacobian are in fact endomorphisms of that Jacobian. If the entries of the period matrices are exact and algebraic over $\Q$, this extra certification is unnecessary. We need only check that putative endomorphisms are correct via a simple linear-algebraic verification. This amounts to given a period matrix $\Pi \in M_{g \times 2g}(\overline{\Q})$ associated to an abelian variety $C^g/\Lambda$, find
$$M \Pi = \Pi R,$$
where $R \in  M_{2g}(\overline{\Q})$ and $M \in M_{g}(\overline{\Q})$. We then know that $R$ is an endomorphism of $\Lambda$, and $M$ is a representation of the endomorphism of the abelian variety on $\C^g$, the tangent space around the origin.

Fortunately, the period matrices associated to cyclically branched covers of $\C\P^1$ via the method described in Section ~\ref{sec:cyclicperiod} are exact. 

\subsection{Examples}
\label{sec:examples}

All curves on which we apply our method are cyclic covers over $\C\P^1.$ Here we describe our examples in terms of their branching indices over $\C\P^1.$ The computation of the period matrices can be found in \url{https://github.com/catherineray/autjac-paper/blob/master/periodmatrices.pdf}. The general version of the period matrix is shown in Section~\ref{sec:cyclicperiod}.

\begin{table}[!hbt]
\caption{Automorphism Groups wrt each of the Principal Polarizations}
\centering
\begin{adjustbox}{ tabular= l | c c r r r, center} \hline
  \shortstack{Curve C} & Genus & $\pi(A)$ & $\Aut(\Jac(C), a_i)$ & $|\Aut(\Jac, a_i)|$ & GAPID \\ \hline\hline
  Klein's quartic & 3 & 2 & $S_4 \times C_2$ & 48 & [48, 48]\\ 
  & & & $GL_3(F_2) \times C_2$ & 336 & [336, 209] \\  \hline %N
Fermat's quartic & 3 & 2 &  $(C_4\wr C_2) \times C_2$ & 64 & [64, 101] \\ %N 
& & & $(C_4^{\text{ }2} \rtimes S_3) \times C_2$ & 192 & [192, 944] \\  \hline
$12(1, 5, 6)$ & 3 & 3 & $D_6$  & 12 & [12, 4]\\ %12, ?
&  & & $C_4 \times S_3$ & 24 & [24, 5] \\
& & & $C_4 \times D_4$ & 32  & [32, 25] \\ \hline 
Bring's curve & 4 & 2 & $C_2^{\text{ }2} \times D_4$ & 32 & [32, 46] \\
& & & $C_2 \times S_5$ & 240 & [240, 189] \\ \hline %N 
I-WP & 4 & 9 & $C_2^{\text{ }4}$ & 16 & [16, 14] \\
 
& & & $C_2^{\text{ } 2} \times C_6$  & 24  & [24, 15] \\
 
& & & $C_2^{\text{ } 2} \times D_4$  & 32 & [32, 46] \\
 
& &   & $C_2^{\text{ }3} \times C_6$ & 48 &  [48, 52] \\
 
%& & & $D_4^{\text{ }2}$ & 64 & [64, 226] \\

& & & $C_2^{\text{ }2}\times S_4$ & 96 & [96, 226] \\

& & & $C_6 \times S_4$ & 144 & [144, 188] \\

& & & $(C_2 \times C_6) \times (C_3 \rtimes D_4)$  & 288 & [288, 1002]  \\

& & & $C_3 \times (((C_6 \times C_2) : C_2) \times D_8)$ & 576 & [576, 7780] \\ 

& & & $C_6 \times (S_3 \times ((C_6 \times C_2) : C_2))$ & 864 & [864, 4523] \\ \hline
\end{adjustbox}
\label{table:tablelabel}
\end{table}

\begin{remark} Note that the University of Bristol's GroupNames database at the time of writing has groups up to order 500 with full names and structure description. In the cases where the order is greater than 500, we use the output of \texttt{StructureDescription(G);} \end{remark}

Table ~\ref{table:plane} is the result of running our program based on ~\cite{numerical}, which is available at \begin{center}\url{https://github.com/catherineray/aut-jac/autplane.sage}\end{center} We use the results in this table together with the precise Torelli theorem (appendix of \cite{Torelli}) to understand the automorphism group of $\Jac(C)$ with respect to the principal polarization induced by the curve $C$, which we apply in Section ~\ref{sec:questions}. 

%Klein's curve 
%Fermat's curve
%12(1,5,6)
%Bring's curve y^5 = x^2(x+1)(x-1)^4, y^5 = x (x - 1)^2 (x + 1)^3
%I-WP

\begin{table}[H]
\caption{Plane Curve Automorphism Groups}
\centering 
\begin{tabular}{ l | l c r r c} \hline
  \shortstack{Curve C} & Plane Curve Model & Genus & Aut(C) & $|$Aut(C)$|$ \\ \hline
  $7(1, 2, 4)$ (Klein's quartic) & $y^7 - x^2(x-1)^4$ & 3 & $GL_3(F_2)$ & 168 \\  
  $8(1, 2, 5)$ (Fermat's quartic) & $y^8 - x(x-1)^2$ & 3 & $C_4^{\text{ }2} \rtimes S_3$ & 96 \\
  $\ast 12(1, 5, 6)$ &  $y^{12} - x(x-1)^5$ & 3 & $C_4 \times S_3$ & 24 \\
  $5(1, 2, 4, 3)$ (Bring's curve) & $y^5 - x (x - 1)^2 (x + 1)^3$ & 4 & $S_5$ & 120 \\ 
  $12(1, 4, 7)$ (I-WP) & $y^{12} - x(x-1)^4$ & 4 & $C_3 \times S_4$ & 72 \\ \hline
\end{tabular}
\label{table:plane} 
\caption*{An $\ast$ indicates that the curve is hyperelliptic}
\centering
\end{table}

\subsection{Questions and Answers on Abelian Varieties with Multiple Principal Polarizations}
\label{sec:questions}

We speak here of polarizations up to auto-equivalence and ask natural questions on Jacobians with multiple principal polarizations, answering all but one of the questions using methods developed in our paper.

We fix some notation. We call $\Aut(A, a_i)$ a symplectic automorphism group of $A$, as the automorphisms respect the principal polarization $a_i$, which is a symplectic form on $A$. Let $\theta_C$ be the canonical principal polarization of $\Jac(C)$ with respect to $C$.

\begin{question} $\Aut(\Jac(C), \theta_C)$ has the highest order of all symplectic automorphism groups of $\Jac(C)$. \end{question}

\begin{answer} This is proven false by $12(1,5,6)$, where $|\Aut(\Jac(12(1, 5, 6)), \theta_{12(1, 5, 6)})| = 24$, but $|\Aut(\Jac(12(1, 5, 6)), a_i)| = 32$ is achieved. It is more dramatically proven false by I-WP, where $|\Aut(\Jac(\text{I-WP}), \theta_{\text{I-WP}})| = 288$, but $|\Aut(\Jac(\text{I-WP}), a_i)|$ achieves $576$ and $864$. \end{answer}

\begin{question} Principal polarizations $a_1$ and $a_2$ are auto-equivalent if and only if they are analytically equivalent. In other words, $$\Aut(X, a_1) \simeq \Aut(X, a_2) \Leftrightarrow a_1 = a_2.$$ \end{question}

\begin{answer} The direction ($\Leftarrow$) is true because $\mc{L}$ and $\mc{M}$ are analytically equivalent if and only if $c_1(\mc{L}) = c_1(\mc{M})$ by [\cite{bl} 2.5.3]. The other direction ($\Rightarrow$) is false.  This is proven false by applying our method to the following two \textit{non-isomorphic} curves with the same \textit{unpolarized} Jacobian from Theorem 1 of \cite{howe1}:
\vspace{-2pt}
$$X: 3y^2 = (2x^2- 2)(16x^4 + 28x^2 + 1)$$ 
\vspace{-15pt}
$$X': -y^2 = (2x^2 + 2)(16x^4 + 12x^2 + 1)$$ 

\noindent which both have $\Aut(\Jac(X), \theta_X)\simeq C_2 \times C_2 \simeq \Aut(\Jac(X'), \theta_{X'})$. \end{answer}

The counterexample to the above question arises by giving an example of a pair non-isomorphic curves $X \nsimeq X'$ with the following two properties: $$\Jac(X) \simeq \Jac(X') \text{  and  }  \Aut(\Jac(X), \theta_X) \simeq \Aut(\Jac(X'), \theta_{X'}).$$ It is natural to ask if this phenomenon occurs for \textit{all} pairs of curves with isomorphic unpolarized Jacobians. % i.e., if the first property implies the second.

\begin{question} Let $C$ and $C'$ be any curves such that $\Jac(C) \simeq \Jac(C')$ as complex varieties, then $$\Aut(\Jac(C), \theta_C) \simeq \Aut(\Jac(C'), \theta_{C'}).$$  \end{question} 

We checked this question on the family of hyperelliptic cases of genus 2 from \cite{howe1} Theorem 1, where it is true. However, there is no reason to expect this to be true in general. Yet, we cannot disprove it easily.

\end{document}